\documentclass[12pt]{article}
\usepackage{latexsym,amssymb,amsmath,amsfonts,amsthm,graphicx}
\newtheorem{theorem}{Theorem}

\newtheorem{lemma}{Lemma}

\theoremstyle{definition}

\def\beq{ \begin{equation} }
\def\eeq{ \end{equation} }
\def\mn{\medskip\noindent}

\def\square{\vcenter{\vbox{\hrule height .4pt
  \hbox{\vrule width .4pt height 5pt \kern 5pt
        \vrule width .4pt} \hrule height .4pt}}}

\def\ep{\epsilon}
\def\ER{Erd\"os-R\'enyi }

\def\CC{{\cal C}}

\def\clearp{}

\begin{document}

\title{Critical behavior of two-choice rules, \\
a class of Achlioptas processes}
	\author{Braden Hoagland and Rick Durrett\thanks{A version of this paper was Braden's senior thesis at Duke. He now works at IMC Trading in Chicago. RD was partially spported by NSF grants DMS 1809967 and 2153429 from the probability program} \\
\small Dept of Math, P.O. Box 90320,
Duke University, Durham NC 27708}

\maketitle

\begin{abstract}
Achlioptas processes are a class of dynamically grown random graphs where on each step several edges are chosen at random but only one is added. The sum rule, product rule, and bounded size rules have been extensively studied. Here we introduce a new collection of rules called two-choice rules. In these systems one first pick $m$ vertices at random from the graph and chooses a vertex $v$ according to some rule based on their cluster sizes. The procedure is then repeated with a second independent sample to pick a vertex $v'$ and we add an edge from $v$ and $v'$. These systems are tractable because the cluster size distribution satisfies an analog of the Smoluchowski equation. We study the critical exponents associated with the phase transitions in five of these models. In contrast to the situation for $d$-dimensional percolation we show that all of the critical exponents can be computed if we know $\beta$, the exponent associated with the size of the giant component. When $\beta=1$ all the critical exponents are the same as for the \ER graph.
\end{abstract}

\clearp

\section{Summary of results} 

We study the phase transitions of a class of dynamically grown graphs that we call {\bf two-choice rules}. We start with a graph with $n$ vertices and no edges. In general, on each step we first pick $m$ vertices $v_1, \ldots, v_m$ at random and choose one, $v$, according to some rule. We then pick a second set of $m$ vertices  $v'_1, \ldots, v'_m$, choose one of them, $v'$, according to the same rule, and then add $(v,v')$ to the graph. 

The \ER graph is the degenerate $m=1$ case of our system: on each step we pick an edge at random and add it to the graph.  To be able to  take the limit as $n\to\infty$ we say that graph with $m$ edges occurs at time $t=2m/n$, so that in the limit as $n\to\infty$ the mean degree of vertices at time $t$ is $t$. In the case of \ER this means the threshold for the existence of a giant component is $t_c=1$. In many papers the graph with $m$ edges occurs at time $t=m/n$ so in the case of \ER this means the threshold for the existence of a giant component is $t_c=1$. 

If we let $\kappa_i$ be the size of the cluster containing $v_i$ then in three of  the rules we consider we  choose a vertex (i) with the smallest $\kappa_i$, (ii) with the largest $\kappa_i$, or (iii)with the median $\kappa_i$ (assuming $m$ is odd). In our fourth system called the Bohman-Frieze vertex rule, we choose an isolated vertex (i.e., $\kappa_i=1$) if there is one in the set, otherwise choose at random from the set. In our fifth example, the {\bf alternative edge rule} \cite{DS-M}, a vertex $v$ is chosen at random and then a secon vertex $v'$ is chosen accroding to the min rule. 

As explained in Section \ref{sec:Achpr} these systems are instances of Achlioptas processes. We now introduce two famous examples. In the {\bf sum rule} one chooses four vertices $v_1,v_2,v_3,v_4$ and connect $v_1,v_2$ if and only if $\kappa_1+\kappa_2 \le \kappa_3+\kappa_4$. In the {\bf product rule} we connect $v_1,v_2$ if and only if $\kappa_1\cdot\kappa_2 \le \kappa_3\cdot\kappa_4$. Achlioptas, D'Souza and Spencer \cite{ADSS} claimed based on simulation that the sum and the product rule had discontinuous phase transitions, but Riordan and Warnkje \cite{RWcont} proved mathematically that the transition was continuous for a very general class of rules. Their proof involved two arguments by contradiction, and hence did not yield much quantitative information about the phase transitions. Here we will study the standard percolation critical exponents for the phase transition of two-choice rules.

\subsection{Critical exponents}

Let $p(s,t)$ be the probability that a randomly chosen vertex belongs to a cluster of size $s$ at time $t$ (in the limit $n\to\infty$). Let $\theta(t)= 1 -\sum_s p(s,t)$, and $\chi_k(t) = \sum_s s^k p(s,t)$ (the value $s=\infty$ is excluded from each of the sums). The exponents $\beta$, $\gamma$, and $\Delta$ are defined by
\begin{align*}
\theta(t) &\approx (t-t_c)^\beta \qquad\hbox{as $t\downarrow t_c$} \\
\chi_1(t) & \approx| t-t_c|^{-\gamma}  \qquad\hbox{as $t\to t_c$} \\
\hbox{for $k\ge 2$}\quad 
\chi_k(t)/\chi_{k-1}(t)& \approx |t-t_c|^{-\Delta}  \qquad\hbox{as $t\to t_c$}
\end{align*}
In the physics literature the meaning of $\approx$ is not precisely defined.
It could be something as weak as 
$$
\frac{\log \theta(t)}{\log(t-t_c)} \to \beta  \qquad \hbox{as $t \downarrow t_c$}
$$
In order to derive relations between exponents we will suppose 
$\theta(t)  \sim C (t-t_c)^\beta$ where $a(t) \sim b(t)$ means $a(t)/b(t) \to 1$.
Our final two exponents concern the behavior near the critical value
\beq
p(s,t) \approx s^{1-\tau} f(s|t-t_c|^{1/\sigma})
\label{nearcr}
\eeq
where $f$ is a {\bf scaling function}.

In Section \ref{sec:ER} we will compute the critical exponents for the \ER model, 
\beq
\beta=1, \quad \gamma=1, \quad \Delta=2, \quad \tau=5/2, \quad \sigma=1/2
\label{ERexp}
\eeq
and show that the scaling relationship \eqref{nearcr} holds. These results are well-known. We include them for completeness.

The most significant difference between our critical exponents and those for $d$-dimensional percolation is that we do not have a {\bf correlation length} $\xi(p)$ that gives the spatial size of  a typical finite cluster. Thi squantity is often defined in terms of the exponential decay of  probability 0 and $x$ are in the same finite cluster
$$
\tau^f(0,x) = P_p( 0 \leftrightarrow x, |\CC_0|< \infty).
$$ 
For example, if $e_1$ is the first unit vector 
$$
\xi(p) = \lim_{n\to\infty} - \frac{1}{n} \log \tau^f(0,ne_1)
$$
A simpler approach taken by Kesten \cite{HKSR} is to define the correlation length by
$$
\xi(p) =\left(  \frac{1}{\chi_1(p)} \sum_y |y|^2 P( 0 \to y; |{\cal C}_0| < \infty) \right)^{1/2}
$$
and the critical exponent $\nu$ by $\xi(p) \approx |p-p_c|^{-\nu}$. Finally on a $d$-dimensional graph we have another exponent $\eta$ called the {\bf anomalous dimension}
$$
P_{cr}( 0 \to x ) \approx |x|^{2-d-\eta}
$$

The percolation critical exponents satisfy 
the following scaling relationships
\beq
\beta = \frac{2\nu}{\delta+1} \qquad
\gamma = 2\nu \frac{\delta-1}{\delta+1} \qquad
\Delta = 2\nu \frac{\delta}{\delta+1} \qquad
\eta = \frac{4}{\delta}
\label{percSR}
\eeq
Kesten \cite{HKSR} established 
provided that the exponents $\delta$ and $\nu$ exist.

\subsection{Models and ODEs}

In the \ER model $p(s,t)$ satisfies
\beq
\frac{\partial  p(s,t)}{\partial t} =  s \sum_{u+v=s}  p(u,t) p(v,t) -2 s p(s,t) 
\label{ERode}
\eeq
In words if the new edge joins a cluster of size $u$ and $v$ with $u+v=s$ we now have $s$ new vertices that are part of clusters of size $s$, while if a cluster of size $s$ is one of the two that merge then the $s$ vertices in that cluster are no longer part of a cluster of size $s$.
If we let $n(s,t) = p(s,t)/s$ be the number of clusters of size s and $K(u,v)=uv$ then \eqref{ERode} becomes the {\bf Smoluchowski equation}
\beq
\frac{\partial  n(s,t)}{\partial t} =  \sum_{u+v=s}  K(u,v) n(u,t) n(v,t)
-2 n(s,t) 
\label{Smoleq}
\eeq
For probabilists the best known reference is Aldous \cite{AldousMFC}, but the survey by
Leyvraz \cite{Leyvraz} is more extensive. In addition to results about existence, uniqueness, and asymptotic behavior, there are exact results for $K(u,v)=1$ (Kingman's coalescent) and $K(u,v)=u+v$ (the additive case).

\mn
{\bf Example 1. Min(m) rule}. da Costa et al, \cite{daCosta-cont,daCosta-critexp,daCosta-scaling}, considered a system in which each choice is made as follows: pick $m$ vertices $v_1 \ldots v_{m}$ with cluster sizes $\kappa_1, \ldots \kappa_{m}$. Let $\kappa_*= \min\{\kappa_1,\ldots \kappa_m\}$ and pick a vertex at random from those with cluster size $\kappa_*$. To write the differential equation for this and other two-choice rules, let $\phi(s,t)$ be the probability of choosing a vertex with cluster size $s$ at time $t$.
\beq
\frac{\partial  p(s,t)}{\partial t} =  s \sum_{u+v=s}  \phi(u,t) \phi(v,t)
-2 s \phi(s,t) 
\label{2cODE}
\eeq
In the case of the min rule if we let $\hat P(u,t) = P(\kappa > u)$ be the tail of the distribution, we have 
\beq
\phi(u,t)= P( \min\{\kappa_1,\ldots \kappa_{m}\} = u) 
= \hat P(u-1,t)^m - \hat P(u,t)^m
\label{dCphi}
\eeq
When $u$ is large it is unlikely that two vertices achieve the minimum so
\beq
\phi(u,t) \sim m p(u,t) \hat P(u,t)^{m-1}
\label{minasy}
\eeq

\mn
{\bf Example 2. Median(2m-1) rule}. Pick $2m-1$ vertices $v_1 \ldots v_{2m-1}$ with cluster sizes $\kappa_i$. Let $\kappa_{med}$ be the median value. Since $2m-1$ is odd the value of the median is unique but it may be achieved by several $\kappa_i$. In this case we pick one of the vertices with the median cluster size at random.
Since the median may occur many times in the sample, it is hard to write an exact formula for $\phi(u,t)$. However if we let $P(u,t) = P( \kappa \le u)$ then for large $u$
\begin{align}
\phi(u,t) \approx c_m P(u-1,t)^{m-1} p(u,t) \hat P(u,t)^{m-1} 
\nonumber\\
\sim P(u,t) \hat P(u,t)^{m-1}
\label{medasy}
\end{align}
where $c_m = (2m-1)!/(m-1)!(m-1)!$.
Thus despite picking from the center of the distribution the asymptotic behavior of $\phi$ will be the same as the Min(m) rule. One can obviously generalize this example to other order statistics, which will then have the same behavior as an apropriate Min rule.

\mn
{\bf Example 3. Max(m) rule}.
If picking the vertices of minimum degree slows percolation then it is natural to guess that picking vertices of maximum degree speeds it up, and ask how this affects the phase transtion. As far as we can tell, this version has not been considered. We pick $m$ vertices $v_1 \ldots v_{m}$ with cluster sizes $\kappa_1, \ldots \kappa_{m}$. Let $\kappa^*= \max\{\kappa_1,\ldots \kappa_m\}$ and pick a vertex at random from those with cluster size $\kappa^*$. 
Using $P(u,t) = P(\kappa \le u)$, we have 
\beq
\phi(u,t)= P( \max\{\kappa_1,\ldots \kappa_{m}\} = u) 
= P(u,t)^m -P(u-1,t)^m
\label{maxphi}
\eeq
As $u \to\infty$ we have
\beq 
\phi(u,t) \sim m p(u,t) \quad\hbox{as $u\to\infty$.}
\label{maxasy}
\eeq
Since $\phi \sim c p$ it is natural (but somewhat naive) to guess based on the similaritiy of \eqref{2cODE} to \eqref{ERode} that the exponents will be the same as for \ER. The next system also has $\phi \sim c p$ and we know that it has the same behavior a \ER

\mn
{\bf Example 4. BF vertex rule}
Here BF stands for Bohman-Frieze \cite{BohFri}. For a discussion of their work see Section \ref{sec:Achpr}. In our version which we call the BF vertex rule,  we pick $m+1$ vertices $v_1 \ldots v_{m+1}$. If any of the first $m$ vertices is isolated ($\kappa_i=1$) then we pick one at random (or pick the first on the list). If not we pick $v_{m+1}$. Though this is not the same as their model, it is a bounded size rule,  so it has the same critical exponents as \ER, see Riordan and Warnke \cite{RW-bs}. The BF vertex rule has choice distribution
\begin{align}
& \phi(1,t) = [1 - (1-p(1,t))^m] +  (1-p(1,t))^m p(1,t) = 1 - (1-p(1,t))^{m+1}
\nonumber\\
& \phi(i,t) = (1-p(1,t))^m p(i,t) \quad \hbox{for $i > 1$}
\label{BFchoice}
\end{align}
In words, we will choose a vertex in  a cluster of size 1, if at least one of the first $m$ vertices is isolated, or if all of them are not isolated and $v_{m+1}$ is, we end up with an isolated vertex. If none of the first $m$ vertices are isolated and $v_{m+1}$ has a cluster of size $i>1$ then the size of the chosen cluster is $i$. 

\subsection{Results}

It has long been known for percolation (see Stauffer \cite{StauST}) that \eqref{nearcr} implies
\begin{align}
\beta & = (\tau-2)/\sigma 
\label{scbeta}\\
\gamma & = (3-\tau)/\sigma 
\label{scgamma}\\
\Delta & = 1/\sigma
\label{scDelta}
\end{align}
In Section \ref{sec:scth} we give the proofs, and state the assumptions on $\tau$ and the scaling function $f$ in \eqref{nearcr} that are needed for these results to be theorems.

Let $\beta_\phi$, $\gamma_\phi$, $\delta_\phi$, and $\tau_\phi$ be the values of the critical exponents when $p(s,t)$ is replaced by $\phi(s,t)$. Since $\phi$ is computed from $p$, we assume that $\phi$ has the scaling property
\beq
\phi(s,t) \approx s^\alpha g(s|t-t_c|^{1/\sigma}
\label{phinearcr}
\eeq
Here $\alpha = 1 - \tau_\phi$ and $\sigma_\phi = \sigma$. As we explain in Section \ref{sec:scth} \eqref{scbeta}--\eqref{scDelta} hold for the exponents with subscript $\phi$, Using \eqref{minasy}, \eqref{medasy}, \eqref{maxasy}, and \eqref{BFchoice} we see that

\begin{lemma} \label{betaphi}
For the Min(m) and Median(2m-1) rules we have
$\beta_\phi=m\beta$ and 
\beq
1-\tau_\phi = 1-\tau + (m-1)(2-\tau) = 2m-1-m\tau 
\label{tauphi}
\eeq
For the Max(m) and BF vertex rules $\beta_\phi=\beta$ and $\tau_\phi=\tau$.
\end{lemma}

Using \eqref{2cODE} we can show, see Lemma \ref{derperc} and \ref{dermean} that if $\Phi = 1 - \sum_v \phi(v,t)$ then
\begin{align}
\partial \theta(t)/\partial t & = 2 \langle s \rangle_\phi \Phi 
\label{betaDE}\\
\partial \langle s \rangle_P/\partial t & =2\langle s \rangle^2_\phi 
- 2\langle s^2\rangle_\phi \Phi
\label{gamDE}
\end{align}
From this it follows that
\begin{align}
 \beta - 1 & = \beta_\phi - \gamma_\phi \
\label{desc1} \\
-\gamma -1 &= -2\gamma_\phi 
\label{desc2a} \\
-\gamma -1 &= \beta_\phi - \gamma_\phi  - \Delta_\phi 
\label{desc2b}
\end{align}
On the second and third lines the two formulas come from $t \uparrow t_c$ and $t \downarrow t_c$ in \eqref{gamDE}.

Finally using a calculation from Appendix E of da Costa et al \cite{daCosta-scaling} we have the remarkable result which relates the two variable in the scaling relation
\beq
\sigma = (\tau-1) + 2\alpha + 2
\label{desc3}
\eeq
The proof begins by letting $f(x) = x^{\tau-1}\tilde f(x)$,  and $g(x) = x^{-\alpha}\tilde g(x)$ in order to rewrite the scaling formulas as
\begin{align*}
p(s,t) & = s^{1-\tau} f(s\delta^{1/\sigma}) 
= \delta^{(\tau-1)/\sigma} \tilde f(s\delta^{1/\sigma}) \\
\phi(s,t) & = s^{\alpha} g(s\delta^{1/\sigma}) 
= \delta^{-\alpha/\sigma} \tilde g(s\delta^{1/\sigma})
\end{align*} 
in order to more easily differentiate with respect to $\delta-|t-t_c|$.
One then substitutes $\phi(v) = \phi(v)-\phi(s)+\phi(s)$ in \eqref{2cODE} of $\partial p(s,t)/\partial t$ so that all of the terms have the large time $s$ in them. When the smoke clears after all of the change of variables, every term on the right0hand side is of order $\delta^{(-2\alpha-2)/\sigma}$ and equating this to the order of the left-hand side gives the result.

Using our scaling relationships we can analyze our examples

\begin{theorem} \label{mainth}
For the Min(m) and Median(2m-1) rules we have
\begin{align*}
\gamma_\phi & = 1 + (m-1)\beta\\
\gamma & = 1 + 2(m-1)\beta\\
\Delta = 1/\sigma & = 1 + (2m-1)\beta \\
\tau - 2 & =  \frac{\beta}{1 + (2m-1)\beta}
\end{align*}
\end{theorem} 

\mn
The first two relationships  are immediate consequences of \eqref{desc1} and \eqref{desc2a}. The other two equations require more algebra so the proof is delayed to Section \ref{sec:pfTh1}. Note that using the four equations in Theorem \ref{mainth}, all of the critical exponents can be computed if $\beta$ is known. In contrast, jn ordinary percolation one needs to know the two exponents that appear in the basic scaling relationship to do this. The reduction from two to one is due to \eqref{desc3}.

\mn
{\bf Example 1. Min(m) rule.} In this case our result was first proved by daCosta et al \cite{daCosta-scaling}.  The results in the first three rows of the table below are based on their simulations. They follow the convention that the $m$th graph occurs at time $m/n$ so we have multiplying their estimates of $t_c$ by 2:
\begin{center}
\begin{tabular}{ccccc}
$m$ & 1 & 2 & 3 & 4 \\
$t_c$ & 1 & 1.8464 & 1.9636 & 1.9898 \\
$\beta$ & 1 & 0.0555 & 0.0104 & 0.0024 \\
$\tau$ & 5/2 & 2.0476 & 2.0099 & 2.0024\\
$1/\sigma$&	2 &	1.1665	&1.0520	& 1.0168 \\
$\gamma_P$ &	1	&1.1110	& 1.0416	& 1.0144\\
$\gamma_Q$ & 1 & 1.0555&	1.0208 & 1.0072
\end{tabular}
\end{center}

\noindent
 The results in this table suggest that as $m\to\infty$
\beq
t_c \to 2 \qquad \tau \to 2 \qquad \sigma, \gamma_P, \gamma_Q \to 1
\label{dCconj}
\eeq
It is clear from the table that $\beta\to 0$ but the question is: how fast? Figure 1
suggests that $\beta \to 0$ exponentially fast. If so then $m\beta \to 0$ and the conjectures about critical exponents in \eqref{dCconj} follow

\begin{figure}[ht]
\begin{center}
\includegraphics[height=3.0in,keepaspectratio]{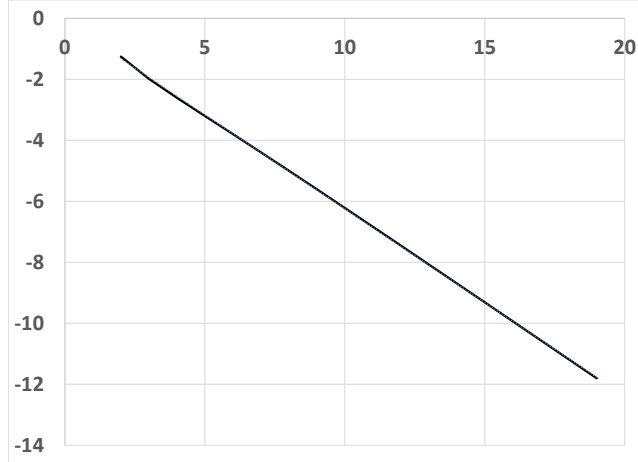}
\caption{Using data from Table I in \cite{daCosta-critexp} which gives estimates of $\beta$ for $m=2$ to 20 and plotting $\log_{10}\beta$ versus $m$ we get a straight line $\log_{10}\beta = -0.615m - 0.0915$. For the fit $r^2 = 0.9999$. } 
\end{center}
\end{figure}

\mn
{\bf Example 2. Median(2m-1) rule.} It would be interesting to know if this models and Min(m) have the same value of $\beta$ and hence all the critical exponents are the same. This conjecture is based on the notion of a {\bf universality class} of models. For example if one considers bond or site percolation on different two-dimensional lattices (rectangular, triangular, hexagonal, etc) then the critical values for the existence of an infinite component are different but the critical exponents are expected to be the same. Likewise the $d$-dimensional contact process ie expected to have exponents that are do not depend on the neighborhood used to define the model and will have the same cirtical exponents as oriented percolation in $(d+1)$ dimensions. 

Theorem \ref{mainth} implies that the Min(m) models are in different universlaity classes (if the values of $\Delta$ are the same for $\ell$ and $m$ then 
$(2\ell-1)\beta_\ell = (2m-1)\beta_m$ the values of $\tau-2$ must be different). Sabbir and Hassam \cite{SabHas} claim that the sum and the product rule defined in the next section have the same critical exponents, but otherwise it seems that within this class of models, the universality classes are small. 

The next result treats our other two models.

\begin{theorem}
In the case of the Max(m) and BF vertex rules we have
\begin{align*}
& \gamma_\phi =  \gamma = 1\\
& \Delta = 1/\sigma  = 1 + \beta\\
& \tau - 2 =  \frac{\beta}{1 +\beta} 
\end{align*}
\end{theorem}

\noindent
Again if we know $\beta$ we can calculate all of the exponents, but this time we want to do more.  As explained in the next section, the BF vertex rule is a bounded size rule,  so it follows from work of Riordan and Warnke \cite{RW-bs} that all the critical exponents are the same \ER  
We think it might be possible to use the proof of Theorem \ref{BBWth} to prove $\tau=5/2$ for the Max(m) rule or any other system in which $\phi \sim c p$.

While in principle we can study two-edge rules in which the choice functions are different, in practice we can only analyze the case in which the first choice is a randomly chosen vertex, 

\mn
{\bf Example 5. Adjacent edge rule.} D'Souza and MItzenmacher \cite{DS-M} considered the case where the first vertex is chosen at random and the second according the min rule. They restricted their attention to the case $m=2$ but here we consider the general situation.

Generalizing the proofs of  Lemmas \ref{derperc} and \ref{dermean} in Section \ref{sec:conODE}
\begin{align*}
\partial \theta(t)/\partial t & = \langle s \rangle_p \Phi 
+ \langle s \rangle_\phi \theta  \\
\partial \langle s \rangle_P/\partial t & =2\langle s \rangle_\phi \langle s \rangle_p
- \langle s^2\rangle_p \Phi - \langle s^2\rangle_\phi \theta
\end{align*}
From this it follows that
\begin{align}
 \beta - 1 & = \min\{ \beta - \gamma_\phi , \beta_\phi - \gamma \}
\label{aesc1} \\
-\gamma -1 &= \gamma_\phi  - \gamma
\label{aesc2a}
\end{align}
where in the second equation we have only used the limit $t \uparrow t_c$.
Generalizing the proof of \eqref{desc3} we can conclude that
\beq
\sigma = \alpha+2
\label{aesc3}
\eeq
Combining our results we have

\begin{theorem}
In the case of the alternative edge rules we have
\begin{align*}
& \gamma_\phi = 1 \\
& \gamma = 1 + (m-1)\beta \\
& \Delta = 1/\sigma  = 1 +m \beta\\
& \tau - 2 =  \frac{\beta}{1 +m\beta} 
\end{align*}
\end{theorem}

\noindent
At the moment we are only able to prove the second result under the assumption that
$ \beta - \gamma_\phi = \beta_\phi - \gamma $. In support of this belief we refer the reader to the sketch of the proof of \eqref{desc3}.

\clearp

\section{Achlioptas processes} \label{sec:Achpr}

In this section we review previous work.

\subsection{Bounded size rules}

Bohman and Frieze \cite{BohFri} were the first to describe a rule that made percolation come later than in the \ER case. Again $G_0=0$ and edges $e_1, e_1'; e_2, e_2'; e_3, e_3' \ldots$ are chosen at random from the set of all possible edges. If on the $i$th step $e_i$ connects two isolated vertices it is added to the graph, otherwise $e_i'$ is added. Using our convention that the graph after $k$ choices is the sate at time to time $2k/n,$ then as $n\to\infty$ They showed that

\begin{theorem}
There is a constant $c_0 > 1.07$ so that the largest component at time $c_0$ has size bounded by $(\log n)^{O(1)}$
\end{theorem}

\noindent
This is an example of a {\bf bounded size rule}. In these systems the decisions about which edge to add is based on the sizes of the components of the vertices chosen, with all of the components of size $\kappa_i \ge K$ being treated the same. Spencer and Wormald \cite{SpeWor} developed a number of important results for the bounded-size case. To state their ``main result let 
$$
S(G) = \frac{1}{n} \sum_v |{\cal C}(v)| = \frac{1}{n} |{\cal C}_i^2|
$$
be the {\it susceptibility}, which in percolation terminology is just the mean cluster size.
They showed, among other things, that

\begin{theorem} The critical value $t_c$ for the emergence of a giant cmponents can be defined by $\lim_{t\uparrow t_c} S(t) = \infty$.
\end{theorem}

As in ordinary percxolation, when $t>t_c$ there is a giant component containing a positive fraction of the sites, while for $t<t_c$ the cluster size distribution has $P(\kappa\ge s ) \le K e^{-cs}$ where $K$ and $c$ depend on $t$. 

In 2013, Bhamidi, Budhirja and Wang \cite{BBW} showed that the behavior of the Bohman-Frieze process near the critical values is the same as in the \ER case. To be precise if we consider the behavior at $t_c + r/n^{1/3}$ with $-\infty < r < \infty$ then as $n\to\infty$ the system converges to the multiplicative coalecsent in which a cluster of size $x$ and a cluster of size $y$ merge at rate $xy$. See Aldous \cite{multcoal} for the corresponding  result for \ER graphs. 

\begin{theorem} \label{BBWth}
There are constants $\alpha,\beta>0$  the Bohman-Frieze process at time  and let $\CC^n_i$ be the clusters writtein in order of decreasing size then then the vector of component sizes 
$$
\frac{\beta^{1/3}}{n^{2/3}} |\CC^n_i|(t_c + \alpha\beta^{2/3}r/n^{1/3})  \qquad 
$$
viewed as a function of $r \in(-\infty,\infty)$ converges to multiplicative coalescent.
\end{theorem}

Riordan and Warnke \cite{RW-bs} have show that ``bounded-size rules share (in a strong sense)  all the features of the \ER phase transiiton.'' Let $L_j$ be the size of the $j$th largest component and let
$$
S_r = \sum_{\CC} |\CC|/n = \sum_{k \ge 1} k^{r-1} N_k/n
$$
where the first sum is over all components and $N_k$ is the number of vertices that belong to components of size $k$. Writing whp (with high probability) for with probaility tending to 1 as $n\to\infty$ they established (see their Theorem 1.1):

\begin{theorem}\label{RWBS}
Let ${\cal R}$ be a bounded-size rule with critical time $t_c>0$. There are rule dependent positive constants $a, A, c, C, \gamma$ and $B_r$, $r\ge 2$ so that for any $\ep\to 0$ with $\ep^3n \to\infty$ as $n\to\infty$

\mn
1. (Subcritical phase) For any fixed $j \ge 1$ and $r \ge 2$
\begin{align*}
&L_j(t_cn-\ep n) \sim C\ep^{-2}\log(\ep^3n) \\
&S_r(t_cn-\ep n) \sim B_r \ep^{-2r+3} 
\end{align*}

\mn
2. (Subcritical phase) For any fixed $j \ge 1$ and $r \ge 2$
\begin{align*}
&L_1(t_cn+\ep n) \sim c\ep n \\
&L_2(t_cn+\ep n) =o(\ep n)
\end{align*}

\mn
3. (Critical regime) Suppose that $k = k(n) \ge 1$ and $\ep=\ep(n)>0$ satisfy $k \le n^\gamma$, $\ep^2 n \le \gamma \log n$, $k \to\infty$, and $\ep^3k\to 0$ 
$$
N_k(t_c \pm \ep n) \sim Ak^{-3/2} e^{-a\ep^2 k} n
$$
\end{theorem}

\subsection{Explosive percolation}

In 2009 Achliptas, D'Souza, and Spencer \cite{ADSS} shouted from the pages of Science that they had found ``Explosive percolation in networks.'' To quote from their paper: 

\begin{quote} \it Here we provide conclusive numerical evidence that unbounded size rules can give rise to discontinuous transitions. For concreteness, we present results for the so-called product rule.\end{quote} 

\noindent
To be precise, and changing notation to avoid conflict with ours, let $L_1(m)$ be the size of the largest cluster when there are $m$ edges, let $m_1$ be the last time $L_1(m)< n^{1/2}$, and let $m_2$ be the first time it is $L(m) > n/2$. Then for large $n$ we have $m_2 - m_1 < n^{2/3}$. 

In 2011 Riordan and Warnke wrote their own Science paper declaring that ``Explosive percolation is continuous.'' Here we will follow the version published the next year in the Annals of Applied Probability. \cite{RWcont}.
The continuity of the phase transition  is proved for a very general set of models

\mn
{\bf $\ell$-vertex rule}. Let $v_m = (v_{m,1}, \ldots v_{m,\ell})$ be an i.i.d.~sequence of vectors in $\{ 1, 2, \ldots n\}^\ell$.  $G_0=\emptyset$ and $G_m$ is obtained from $G_{m-1}$ by adding a possibly empty set of edges $E_m$. $E_m \neq \emptyset$ if the $v_{i,m}$ are in distinct components.

\mn
{\bf Theorem 1 in \cite{RWcont}.} {\it Let ${\cal R}$ be an $\ell$ vertex rule with $\ell \ge 2$. Given any functions $h_L(n)$ and $h_m(n)$ that are $o(n)$ and any constant $\delta>0$ the probability there are $m_1$ and $m_2$ with $L_1(m_1) \le h_L(n)$ and $L_1(m_2) \ge \delta n$ and $m_2 \le m_1 + h_M(n)$ tends to 0 as $n\to\infty$.}

\medskip
To obtain more detailed results they reduced the collection of models.
An $\ell$-vertex rule is  said to be {\bf merging} if whenever $\CC_1$ and $\CC_2$ are distinct components with $|\CC_1|,|\CC_2| \ge \ep n$ then in the next step we have probability $\ge \ep^\ell$ of joining $\CC_1$ to $\CC_2$.

\mn
{\bf Theorem 7 in \cite{RWcont}.} {\it Let ${\cal R}$ be a merging $\ell$-vertex rule. Given constant $0 \le a < b$ and any function $h_m(n) = o(n)$ the probability that there are $m_1$ and $m_2$ with $L_1(m_1) \le an$ and
$L_1(m_2) \ge bn$ and $m_2 \le m_1 + h_m(n)$ tends to 0 as $n\to \infty$.}

\mn
This result gives a proof of Spencer's ``no two giants'' conjecture. If at some time there is a component of size $>\ep n$ in addition to the giant component then when they merge the size will increase by at least $\ep n$ contradicting the last result

The results in \cite{RWcont} are very general and their proofs are ingenious. However, the proof of their Theorem 1 given on page 1457 is based on an argument by contradiction, and it does not seem possible to extract quantitative result. Our goal here is to prove results about the critical behavior of some of these processes, specifically about the values of critical exponents defined in the next section.

\clearp

\section{Proofs of the Theorems} 

\subsection{Proof of Theorem 1}\label{sec:pfTh1}

Lemma \ref{betaphi} implies $\beta_\phi=m\beta$, $\nu=m$, 
so using \eqref{desc1}
$$
\beta-1 = m\beta - \gamma_\phi \quad\hbox{or}\quad
 \gamma_\phi=1+(m-1)\beta
$$
Using the first equality in \eqref{desc2a} $-\gamma - 1 = - 2\gamma_\phi$ so
$$
\gamma = 2 \gamma_\phi + 1 = 1 + 2(m-1)\beta
$$
To prove the next result we note that Lemma \ref{betaphi} implies
$$
\alpha = 2m-1 - m\tau
$$ 
and using \eqref{desc3} $\sigma=(\tau-1) +2\alpha +2$
\begin{align}
\sigma  &= \tau-1 + 2(2m-1-m \tau) + 2
\nonumber \\
& = 1 + 2(2m-1) - (2m-1) \tau =1 - (2m-1)(\tau-2)
\label{sigeq}
\end{align}
Dividing each side by $\sigma$ and using \eqref{scbeta} $(\tau-2)/\sigma=\beta$ gives
$$
1/\sigma  = 1 + (2m-1)\beta 
$$
Rearranging \eqref{sigeq} then using the last result twice
$$
\tau - 2 = \frac{1-\sigma}{2m-1} = \frac{(1/\sigma -1)/(2m-1)}{1/\sigma}
=  \frac{\beta}{1 + (2m-1)\beta} 
$$
which completes the proof.

\subsection{Proof of Theorem 2}\label{sec:pfTh2}

Since $\phi$ is Min(m), Lemma \ref{betaphi} implies $\beta_\phi=m\beta$ so \eqref{desc1} implies
$\gamma_\phi=1$ and \eqref{desc2a} implies $\gamma=1$. Using 
\eqref{scbeta}, \eqref{scgamma} and \eqref{scDelta} we have
$$
\Delta = 1/\sigma = \beta + \gamma = 1 + \beta
$$
Using \eqref{scbeta}
$$
\tau - 2 = \frac{\beta}{1/\sigma} = \frac{\beta}{1+\beta}
$$
which completes the proof.

\subsection{Proof of Theorem 3}\label{sec:pfTh3}

Lemma \ref{betaphi} implies $\beta_\phi=\beta$. \eqref{aesc2a} implies $\gamma_\phi=1$. The $\phi$ versions of \eqref{scbeta}, \eqref{scgamma} and \eqref{scDelta} imply
$$
\Delta=1/\sigma = \beta_\phi + \gamma_\phi = 1 +m \beta
$$
\eqref{aesc3} and \eqref{sigeq} imply 
$$
\sigma=\alpha+2 = 2m+1 - m \tau = 1 -m(\tau-2)
$$
Rearranging gives
$$
\tau - 2 = \frac{1-\sigma}{m} = \frac{(1/\sigma-1)/m}{1/\sigma}
= \frac{\beta}{1+m\beta}
$$
To get the final result we have to assume that the two quantities on the right hand side of \eqref{aesc1} are equal, i.e. $\beta - \gamma_\phi = \beta_\phi - \gamma$,
in order to conclude that
$$
\gamma = 1 + (m-1)\beta
$$

\clearp

\section{Proofs of scaling relations} 

\subsection{Results that follow from \eqref{nearcr}} \label{sec:scth}

To have rigorous result we state explicitly what we assume about $\tau$ and about the scaling function. Throughout we only use the scaling relationship
$$
p(s,t) = s^{1-\tau} f(s|t-t_c|^{1/\sigma})
$$
Since $\phi(s,t) = s^{1-\tau_\phi} g(s|t-t_c|^{1/\sigma})$ the $\phi$ versions of the
scaling relationships follow immediately.

\begin{lemma}
If $2 < \tau < 3$ and $f$ is bounded, and is Lipschitz continuous at 0. 
$$
\beta = (\tau-2)/\sigma
$$
\end{lemma}

\begin{proof} Using \eqref{nearcr} and replacing sum by integration.
$$
\theta(t) \approx \int_1^\infty s^{1-\tau} [f(0) - f(s \delta^{1/\sigma})] \, ds
$$
Changing variables $s= x \delta^{-1/\sigma}$, $ds = \delta^{-1/\sigma} dx$
the above
$$
= \delta^{(\tau-2)/\sigma} \int_{\delta^{1/\sigma}}^\infty
x^{1-\tau} [f(0) - f(x)] \, dx
$$
Since $f$ is bounded and $\tau > 2$ the integral over $[1,\infty)$ is finite.
Since $f$ is Lipschitz continuous the integrand is $ \le Cx^{2-\tau}$ near 0. Since $\gamma<3$, the integral over $[0,1]$ is finite, and it follows that
 $\theta(t) \sim C  \delta^{(\tau-2)/\sigma}$.
\end{proof}

The next result establishes \eqref{scgamma} and \eqref{scDelta}

\begin{lemma} If $2< \gamma < 3$ and $\int_1^\infty x^m f(x) \, dx < \infty$ for all $m$ then for all $\rho\ge 1$
$$
\Gamma(r) = (r+2-\tau)/\sigma
$$
It follows that for all integers $k \ge 2$, 
$$
\Delta_k = \Gamma(k)-\Gamma(k-1) = 1/\sigma.
$$
\end{lemma}

\begin{proof}
\beq
E|\CC_x|^r = \int_1^\infty s^{r+1-\tau} f(s \delta^{1/\sigma}) \, ds
\label{rhomom}
\eeq
Changing variables $s= x \delta^{-1/\sigma}$, $ds = \delta^{-1/\sigma} dx$
the above
$$
= \delta^{(\tau-2-r)/\sigma} \int_{\delta^{1/\sigma}}^\infty
x^{\rho+1-\tau} f(x) \, dx
$$
The assumption $\int_1^\infty x^m f(x) \, dx < \infty$ for all $m$ implies that the integral over $[1,\infty)$ is finite. Since $\tau<3$ and $\rho\ge 1$  the integral over $[0,1]$ is finite
 \end{proof}

\subsection{Proofs of ODE consequences} \label{sec:conODE}

\mn
We will now derive two consequences of the ODE for two-choice rules \eqref{genODE}. Let $\theta(t) = 1 - \sum_{i=1}^\infty p(1,t)$ be the fraction of vertices in giant components. The first result establishes \eqref{desc1}

\begin{lemma} \label{derperc}
Let $\Phi = 1 - \sum_v \phi(v,t)$
and $\langle s \rangle_\phi = \sum_s s \phi(s,t)$
$$
\frac{\partial  \theta(t)}{\partial t} =  2\langle s \rangle_\phi \Phi
$$
\end{lemma}

\begin{proof} $\theta(t) = 1 - \sum_s P(s,t)$ so we have 
\begin{align*}
\frac{\partial \theta}{\partial t} & = - \sum_s \frac{\partial P}{\partial t} \
 = -\sum_s  s \sum_{u+v=s} \phi(u,t) \phi(v,t) + \sum_s 2s \phi(s,t) \\
& = -\sum_u  \sum_v (u+v) \phi(u,t) \phi(v,t) + \sum_s 2s \phi(s,t)  \\
& = -2 \sum_u u \phi(u,t) \sum_v \phi(v,t) + \sum_s 2s \phi(s,t)  \\
& = 2\sum_s s \phi(s,t) \left[ 1- \sum_v \phi(v,t) \right] =  2\langle s \rangle_\phi \Phi
\end{align*}
which proves the desired result. \end{proof}

While the proof is fresh on the readers mind we will prove the result for
the alternate edge rule.

\begin{lemma} \label{AEd1}
$$
\frac{\partial  \theta(t)}{\partial t} =  
\langle s \rangle_p \phi + \langle s \rangle_\phi \theta
$$
\end{lemma}

\begin{proof} Computing as above using the alterantive edge ODE
\begin{align*}
\frac{\partial \theta}{\partial t}
& = -\sum_u  \sum_v (u+v) p(u,t) \phi(v,t) 
+ \sum_s  sp(s,t) + \sum_s s \phi(s,t)  \\
& = -\sum_u u p(u,t) \sum_v \phi(v,t) + \sum_u u p(u,t)  \\
& - \sum_v v \phi(v,t) \sum_u \phi(u,t) + \sum_v v \psi(u,t) 
\end{align*}
which proves the desired result. \end{proof}

\mn
We now prove the second result

\begin{lemma} \label{dermean}
$\partial\langle s \rangle_p/\partial t
= 2\langle s \rangle ^2_\phi - 2\Phi \langle s^2\rangle_\phi$
\end{lemma}

\begin{proof} Multiplying by $s$ and then summing 
\begin{align*}
\frac{\partial\langle s \rangle_p}{\partial t} & = 
\sum_s  s^2 \sum_{u+v=s} \phi(u,t) \phi(v,t) - \sum_s 2s^2 \phi(s,t) \\
 & = \sum_u \sum_v (u+v)^2 \phi(u,t) \phi(v,t) - \sum_s 2s^2 \phi(s,t) \\
 & = 2\sum_u u \phi(u,t) \sum_v v \phi(v,t) + 2
\sum_u u^2 \phi(u,t) \sum_v  \phi(v,t) - \sum_s 2s^2 \phi(s,t)\\
& = 2\langle s \rangle ^2_\phi - 2\Phi \langle s^2\rangle_\phi
\end{align*}
which gives the desired result. 
\end{proof}

Again our next step is to extend to the alternative edge rule.

\begin{lemma} \label{AEd2}
$\partial\langle s \rangle_p/\partial t
= 2\langle s \rangle_p\langle s \rangle_\phi 
- \Phi \langle s^2\rangle_p- \theta \langle s^2\rangle_\phi$
\end{lemma}

\begin{proof} Multiplying by $s$ and then summing 
\begin{align*}
\frac{\partial\langle s \rangle_p}{\partial t} 
 & = \sum_u \sum_v (u+v)^2 p(u,t) \phi(v,t)
 - \sum_s s^2 p(s,t)  - \sum_s s^2 \phi(s,t) \\
 & = 2\sum_u u p(u,t) \sum_v v \phi(v,t) \\
& + \sum_u u^2 p(u,t) \sum_v  \phi(v,t) - \sum_s s^2 p(s,t)\\
& + \sum_v v^2 \phi(v,t) \sum_u  p(v,t) - \sum_s s^2 \phi(s,t)
\end{align*}
which gives the desired result. 
\end{proof}

\subsection{A computation from Appendix E of \cite{daCosta-scaling}}
\label{sec:appE}

Their argument is for the 2 choice min rule. To generalize we suppose
$$
\phi(s,t) = s^\alpha  g(s\delta^{1/\sigma}) 
$$
It is convenient to let $f(x) = x^{\tau-1}\tilde f(x)$,  and $g(x) = x^{-\alpha}\tilde g(x)$ in order to rewrite the scaling formulas as
\begin{align}
p(s,t) & = s^{1-\tau} f(s\delta^{1/\sigma}) 
= \delta^{(\tau-1)/\sigma} \tilde f(s\delta^{1/\sigma})
\label{Psc2} \\
\phi(s,t) & = s^{\alpha} g(s\delta^{1/\sigma}) 
= \delta^{-\alpha/\sigma} \tilde g(s\delta^{1/\sigma})
\label{Qsc2}
\end{align} 

To eliminate the contribution from small values of $u$ and $v$ 
we substitute $\phi(u)=\phi(u)-\phi(s) + \phi(s)$ 
and $\phi(s-u)=\phi(s-u)-\phi(s) + \phi(s)$ in \eqref{genODE} to get
\begin{align*}
&\frac{\partial p(s)}{\partial t}  = s\sum_{u=1}^{s-1} \phi(u)\phi(s-u) - 2s \phi(s) \\
 &=s\sum_{u=1} ^{s-1}[\phi(u)-\phi(s)][\phi(s-u) - \phi(s)] \\
&+ s \sum_{u=1}^{s-1} \phi(s)\phi(s-u) + s \sum_{u=1}^{s-1} \phi(u)\phi(s) 
 - s \sum_{u=1}^{s-1} \phi(s)^2 - 2 s \phi(s)
\end{align*}
where we have dropped the dependence on $t$ to simplify the formulas.

Rearranging gives
\begin{align*}
\frac{\partial p(s)}{\partial t} 
& =s \sum_{u=1}^{s-1}[\phi(u) -\phi(s)][\phi(s-u)-\phi(s)] \\
 &  + 2s\phi(s)\left[ 1 - \sum_{u=s}^\infty \phi(u) \right] 
 -s(s-1) \phi(s)^2 -2s \phi(s)
\end{align*}
Canceling the two $2s\phi(s)$ and writing in integral form
\begin{align}
\frac{\partial p(s)}{\partial t} & = -s^2 \phi^2(s) 
- 2s\phi(s)\int_{s}^\infty \phi(u) \, du
\nonumber\\
& + s \int_{0}^{s}[\phi(u) -\phi(s)] [\phi(s-u)-\phi(s)] \, du
\label{dCDE2}
\end{align}
Taking $\delta=t_c-t$ and using $s \delta^{-1/\sigma} = x$ with
$p(s,t)= \delta^{(\tau-1)/\sigma} \tilde f(s\delta^{1/\sigma})$
\begin{align}
\frac{\partial  p}{\partial t} 
& = - \frac{ \tau-1}{\sigma} \delta^{-1 +(\tau-1)/\sigma}
 \tilde f(s\delta^{1/\sigma})
 -  \delta^{(\tau-1)/\sigma} \tilde f'(s\delta^{1/\sigma}) 
\frac{s}{\sigma} \delta^{-1+1/\sigma} 
\nonumber\\
& = \left[ - \frac{ \tau-1}{\sigma} \tilde f(x)  - \frac{x}{\sigma}  \tilde f'(x)  \right]
\cdot  \delta^{-1 +(\tau-1)/\sigma} 
\label{scLHS}
\end{align}
Using $\phi(r)= \delta^{-\alpha/\sigma} \tilde g(r\delta^{1/\sigma})$
the right-hand side of \eqref{dCDE2} becomes
\begin{align*}
&= -s^2  \delta^{-2\alpha/\sigma} \tilde g(s\delta^{1/\sigma})^2
-2s  \delta^{-2\alpha/\sigma} \tilde g(s\delta^{1/\sigma})  \int_s^\infty du\,  \tilde g(u\delta^{1/\sigma}) \\
& + s  \delta^{-2\alpha/\sigma}  \int_{0}^{s} du
 \,[ \tilde g(u\delta^{1/\sigma})  - \tilde g(s\delta^{1/\sigma}) ]
\cdot[ \tilde g((s-u)\delta^{1/\sigma})  - \tilde g(s\delta^{1/\sigma}) ]
\end{align*}
using $s=x\delta^{-1/\sigma}$, $u=y\delta^{-1/\sigma}$, and 
$du = dy\, \delta^{-1/\sigma}$
\begin{align}
& = -x^2  \delta^{(-2\alpha - 2)/\sigma} \tilde g(x)^2
-2  x \delta^{(-2\alpha-2)/\sigma} \int_x^\infty dy\,  \tilde g(y)
\nonumber \\
& + x  \delta^{(-2\alpha-2)/\sigma} \int_{0}^{x} dy
 \,[ \tilde g(y)  - \tilde g(x) ]
,[ \tilde g(x-y)  - \tilde g(x) ]
\label{scRHS}
\end{align}
In order for \eqref{dCDE2} to hold the powers of $\delta$ must be the same, i.e.
$$
- 1 + \frac{\tau-1}{\sigma} = \frac{-2\alpha-2}{\sigma}
$$
Rearranging gives 
\beq
\sigma = (\tau-1)+  2\alpha + 2
\label{ODEsc3}
\eeq

\subsection{Proof for adjacent edge rule}

To eliminate the contribution from small values of $u$ and $v$ in \eqref{genODE}
we again substitute $p(u)=p(u)-p(s) + p(s)$ 
and $\phi(s-u)=\phi(s-u)-\phi(s) + \phi(s)$ to get
\begin{align*}
&\frac{\partial p(s)}{\partial t}  = s\sum_{u=1}^{s-1} p(u)\phi(s-u) 
- s p(s) - s \phi(s) \\
 &=s\sum_{u=1} ^{s-1}[p(u)-p(s)][\phi(s-u) - \phi(s)] \\
&+ s \sum_{u=1}^{s-1} p(s)\phi(s-u) + s \sum_{u=1}^{s-1} p(u)\phi(s) 
 - s \sum_{u=1}^{s-1}p(s) \phi(s) - s p(s) - s \phi(s)
\end{align*}
Rearranging gives
\begin{align*}
\frac{\partial p(s)}{\partial t} 
 &=s\sum_{u=1} ^{s-1}[p(u)-p(s)][\phi(s-u) - \phi(s)] \\
&+ sp(s)  \sum_{v=1}^{s-1} \phi(v) + s\phi(s)  \sum_{u=1}^{s-1} p(u)
 - s(s-1)p(s) \phi(s) - s p(s) - s \phi(s)
\end{align*}
Canceling and writing in integral form
\begin{align}
\frac{\partial p(s)}{\partial t} =
& - s^2p(s) \phi(s)
- sp(s)  \int_s^\infty  \phi(u) \,du - s\phi(s)  \int_s^\infty  p(u) \, du
\nonumber \\
 &+s\int_0^s [p(u)-p(s)][\phi(s-u) - \phi(s)] \,du 
\label{AEeq}
\end{align}
The formula for the left hand side given in \eqref{scLHS} remains the same.
 Using $p(r)= \delta^{(\tau-1)/\sigma} \tilde f(r\delta^{1/\sigma})$
$\phi(r)= \delta^{-\alpha/\sigma} \tilde g(r\delta^{1/\sigma})$ 
the right-hand side is
\begin{align*}
&= -s^2  \delta^{(\tau-1-\alpha)/\sigma} \tilde f(s\delta^{1/\sigma})
\tilde g(s\delta^{1/\sigma})\\
& -s  \delta^{(\tau-1-\alpha)/\sigma} \tilde f(s\delta^{1/\sigma})  \int_s^\infty du\,  
\tilde g(u\delta^{1/\sigma}) 
-s  \delta^{(\tau-1-\alpha)/\sigma} \tilde g(s\delta^{1/\sigma})  \int_s^\infty du\,  
\tilde f(u\delta^{1/\sigma}) \\
& + s  \delta^{(\tau-1-\alpha)/\sigma}  \int_{0}^{s} du
 \,[ \tilde f(u\delta^{1/\sigma})  - \tilde f(s\delta^{1/\sigma}) ]
\cdot[ \tilde g((s-u)\delta^{1/\sigma})  - \tilde g(s\delta^{1/\sigma}) ]
\end{align*}
Using $s=x\delta^{-1/\sigma}$, $u=y\delta^{-1/\sigma}$
$du = dy\, \delta^{-1/\sigma}$
\begin{align*}
&= -x^2  \delta^{(\tau-1-\alpha-2)/\sigma} \tilde f(x)\tilde g(x)\\
& -x  \delta^{(\tau-1-\alpha-2)/\sigma} \tilde f(x)  \int_x^\infty dy\,  \tilde g(y) 
-x  \delta^{(\tau-1-\alpha-2)/\sigma} \tilde g(x)  \int_s^\infty dy\,  
\tilde f(y) \\
& + x  \delta^{(\tau-1-\alpha-2)/\sigma}  \int_{0}^{x} dy
 \,[ \tilde f(y)  - \tilde f(x) ]
\cdot[ \tilde g(x-y)  - \tilde g(x) ]
\end{align*}

In order for \eqref{AEeq} to hold the powers of $\delta$ must be the same, i.e.
$$
- 1 + \frac{\tau-1}{\sigma} = \frac{\tau-1-\alpha-2}{\sigma}
$$
Rearranging gives 
\beq
\sigma = \alpha + 2
\label{AEsc3}
\eeq

\clearp

\section{\ER model} \label{sec:ER}

\subsection{Survival probability}

Let $x$ be a randomly chosen vertex and let $Z_m$ be the number of vertices at distance $m$ from $x$. When $m =o(\log n)$ the cluster containing $x$ is whp a tree and $Z_m$ is a branching process in which each individual in generation $m$ has a Poisson($\mu$) number of offespring.

Consider a branching process with offspring distribution $r_k$ with mean $\mu>1$ and finite second moment. Let $\phi(z) = \sum_{k=0}^\infty r_k z^z$ be the generating function. If $r_K$ is Poisson($\mu$) then
\beq
\phi(z) = \sum_{k=0}^\infty e^{-\mu} \frac{\mu^k|}{k!}z^k 
\exp(-\mu(1-z))
\label{Poissongf}
\eeq

Let $\rho$ be probability the system dies out. Breaking things down according to the number of children in the first generation
$$
\rho = \sum_{k=0}^\infty r_k \rho^k = \phi(\rho)
$$
$\rho(1)=1$ is a trivial solution. $\rho$ is the unique solution of $\phi(\rho)=\rho$ in $[0,1)$.
$$
\phi(1)=1 \qquad \phi'(1) = \sum_{k=0}^\infty k r_k = \mu \qquad
\phi''(1) = \sum_{k=0}^\infty k(k-1) r_k = \mu_2
$$
If $\mu$ is close to 1 then $\rho$ will be close to 1. Ignoring a few details, if $x$ is close to 1 expanding $\phi$ in power series around 1 gives
$$
\phi(1-x) = 1 - \mu x + \mu_2 x^2/2
$$
so for a fixed point we want
$$
(\mu-1)x =  \mu_2 x^2/2
$$
or $x = 2(\mu-1)/\mu_2$. If we let $\theta(\mu) = P( |{\cal C}_x|=\infty)$ which is the same as the fraction of vertices in the giant component
\beq
\theta(\mu) \sim \frac{2}{\mu_2}(\mu-1)
\label{ERbeta}
\eeq
so the critical exponent $\beta =1$.

\subsection{Mean cluster size.}

Let ${\cal C}_x$ be the cluster containing $x$. If $\mu < 1$
\beq
E|{\cal C}_x| = 1 + \mu + \mu^2 + \cdots = \frac{1}{1-\mu}.
\label{ERgamma}
\eeq
so $\gamma=1$. In the supercritical regime we consider
$$
E(|{\cal C}_x|; |{\cal C}_x|<\infty)
$$
The cluster size is the same as the total progeny in a supercritical branching process conditioned to die out.

\begin{theorem}
\label{condtodie}
A supercritical branching process conditioned to become extinct
is a subcritical branching process. If the original offspring
distribution is Poisson($\mu$) with $\mu>1$ then the
conditioned one is Poisson($\mu\rho$) where $\rho$ is
the extinction probability.
\end{theorem}

\begin{proof}
Let $T_0 = \inf\{ t : Z_t = 0 \}$
and consider $\bar Z_t = (Z_t | T_0 < \infty)$. 
To check the Markov property for $\bar Z_t$ note that the
Markov property for $Z_t$ implies:
$$
P( Z_{t+1} = z_{t+1}, T_0 < \infty |Z_t = z_t, \ldots Z_0 = z_0 )
= P( Z_{t+1} = z_{t+1}, T_0 < \infty |Z_t = z_t)
$$
To compute the transition probability for $\bar Z_t$, observe that if $\rho$ is the extinction
probability then $P_x(T_0 < \infty) = \rho^x$. Let $p(x,y)$ be the transition
probability for $Z_t$. Note that the Markov property implies
$$
\bar p(x,y) = \frac{ P_x ( Z_1 = y, T_0 < \infty) } { P_x (T_0 < \infty )}
= \frac{ P_x ( Z_1 = y ) P_y( T_0 < \infty) } { P_x (T_0 < \infty )}
= \frac{ p(x,y) \rho^y } {\rho^x}
$$
Taking $x=1$ and computing the generating function
\beq
\sum_{y=0}^\infty \bar p(1,y) z^y = \rho^{-1} \sum_{y=0}^\infty p(1,y) (z\rho)^y
= \rho^{-1} \phi(z\rho)
\label{condgf}
\eeq
where $p_y = p(1,y)$ is the offspring distribution. 

$\bar p_y = \bar p(1,y)$ is the distribution of the size of the family of an individual,
conditioned on the branching process dying out. If we start with $x$ individuals
then in $Z_n$ each gives rise to an independent family. In $\bar Z_n$ each
family must die out, so $\bar Z_n$ is a branching process with offspring
distribution $\bar p(1,y)$. To prove this formally observe that
$$
p(x,y)  = \sum_{j_1, \ldots j_x \ge 0, j_1 + \cdots + j_x = y} p_{j_1} \cdots p_{j_x} 
$$
Writing $\sum_*$ as shorthand for the sum in the last display
$$
\frac{p(x,y) \rho^y}{\rho^x} 
= \sum_* \frac{p_{j_1} \rho^{j_1}}{\rho} \cdots \frac{ p_{j_x} \rho^{j_x}}{\rho} 
= \sum_* \bar p_{j_1} \cdots \bar p_{j_x}
$$
In the case of the Poisson($\mu$) distribution $\phi(z) = \exp(\mu(z-1))$ so if $\lambda>1$, so using \eqref{condgf}
$$
\frac{ \phi(z\rho) } {\rho } 
= \frac{\exp(\mu(z\rho-1))} { \exp(\mu(\rho-1)) }
= \exp( \mu\rho (z-1) )
$$
which completes the proof that the conditioned process is a Poisson($\mu\rho$) branching process
\end{proof}

Using the result for the subcritical case in \eqref{ERgamma}
$$
E(|{\cal C}_x| \, | \, |{\cal C}_x|<\infty) = \frac{1}{1-\mu\rho}
$$
Since $\mu_2=\mu^2$ for Poisson, \eqref{ERbeta} shows that if $\mu$ is close to 1
$$
\rho \approx 1 -\frac{2(\mu-1)}{\mu^2}
$$
so we have
$$
1 - \mu\rho = 1 - \mu + \frac{2(\mu-1)}{\mu} 
= \frac{2(\mu-1)-\mu(\mu-1)}{\mu}
$$
From this we see that
$$
E(|{\cal C}_x|\,  |\, |{\cal C}_x|<\infty) =
\frac{1}{1-\mu\rho} 
\sim \frac{1}{(1-\mu)}
$$
thus the asymptotic behavior of the mean cluster size as $\mu\downarrow 1$ is exactly the same as  as $\mu\uparrow 1$

\subsection{Higher moments}

Although the connection with clusters in Erd\"os-R\'enyi random graps and branching processes is intuitive, it is more convenient technically to expose the cluster one site at a time to obtain something that can be approximated by a random walk. Let $\eta_{x,y}=\eta_{y,x}=1$ if there is an edge between $x$ and $y$.
Let  $R_t$ be the set of removed sites, $U_t$ be the unexplored sites and $A_t$ is the set of active sites. Initially $R_0 = \emptyset$, $U_0=\{2, 3, \ldots, n\}$, and $A_0=\{1\}$.  If $A_t\neq\emptyset$, pick $i_t$ at random from $A_t$  let 
\begin{align}
R_{t+1} & = R_t \cup \{i_t\} \cr
A_{t+1} & = A_t - \{i_t\} \cup \{ y \in U_t : \eta_{i_t,y} = 1 \} \cr
U_{t+1} & = U_t - \{ y \in U_t : \eta_{i_t,y} = 1 \} 
\label{RAUdef}
\end{align}

At time $\tau = \inf\{ t : A_t=\emptyset\}$ we have found all the sites in the cluster and the process stops. $|R_t|=t$ for all $t \le \tau$, so the cluster size is $\tau$. If $|A_n|>0$ and the number of rermoved sites is small
$$
S_{n+1} \approx S_n - 1 + \hbox{Poisson}(\mu)
$$
To have the process defined for all time let $\xi_1, \xi_2, \ldots$ be i.i.d.~$-1 + \hbox{Poisson}(\mu)$ and $S_{n+1} = S_n + \xi_{n+1}$.

\mn
We begin by computing the moment generating function of $\xi_i$
\beq
\psi(\theta)  =  e^{-\theta}  \sum_{m=0}^\infty e^{-\mu} \frac{\mu^m}{m!}
 = \exp( - \theta + \mu (e^{\theta}-1) ) 
\label{mgfPoiss}
\eeq
$\exp(\theta S_t)/\phi(\theta)^t$ is a nonnegative martingale, so using the optional stopping theorem for the nonnegative supermartingale $M_t = \exp(\theta S_t)/\psi(\theta)^t$, see e.g., (7.6) in Chapter 4 of Durrett (2004)
\beq
M_0= e^{\theta} \ge E(\psi(\theta)^{-\tau})
\label{opstop}
\eeq

$\psi'(0) = E\xi_i = 1-\mu$ so if $\mu<1$ then $\psi(\theta)<1$ when $\theta>0$ is small.  To optimize we note that the derivative 
$$
\frac{d}{d\theta} (-\theta + \mu (e^{\theta}-1)) = -1 + \mu e^{\theta} =0
$$ 
when $\theta_1=-\log\mu$. At this point $e^{\theta_1}=1/\mu$ and 
$$
\psi(\theta_1) = \exp( \log(\lambda) + 1 - \lambda ) \equiv e^{-\alpha} < 1
$$
Since $\psi(\theta_1^{-1}=e^\alpha >1$, using Chebyshev's inequality with \eqref{opstop}
$$
e^{\alpha m} P( \tau \ge m )  = \psi(\theta_1)^{-m}  P( \tau \ge m ) 
\le E(\psi(\theta_1)^{-\tau}) \le e^{\theta_1}
$$
One particle dies on each time step so $1+\xi_1 + \cdots + \xi_{\tau} = \tau$ and we have
\beq
P\left( \max_{0 \le n \le \tau} S_n \ge m \right)  \le e^{-m \alpha} /\mu
\label{rwdom}
\eeq

There are several other martingale associated with a random walk. Perhaps the simplest is $S_n - n(\mu-1)$. Using the domination in \eqref{rwdom} we can
conclude that
$$
1 = S_0 = E(S_\tau - (\mu-1)\tau) = (1-\mu) E\tau
$$
so the expected cluster size is $E\tau=1/(1-\mu)$.

If $T_n$ is a random walk in which steps have mean 0 and variance $\sigma^2$
then $T_n-n \sigma^2$ is a martingale. Applying this result to $T_n = S_n - n(\mu-1)$ and recalling that $-1+ \hbox{Poisson}\mu$) has variance $\mu$ we see that
$$
(S_n - n(\mu-1))^2 - \mu n
$$
is a martingale. Using the domination we can stop at time $\tau$ and conclude
\begin{align*}
1 & = E(S_\tau - \tau(\mu-1))^2 - \mu E\tau \\
& = (\mu-1)^2 E\tau^2 - \frac{\mu}{1-\mu}
\end{align*}
Rearranging we have
\beq
E\tau^2 = \frac{1}{(1-\mu)^3}
\label{ER3rdmom}
\eeq
Since $E\tau=1/(1-\mu)$ the critical exponent $\Delta_2=2$.

\subsection{Cluster size at criticality.}

 Lemma 2.7.1 in \cite{RGD} states the following

\begin{theorem}  Let $\alpha(\lambda) = \lambda-1-\log(\lambda)$. If $k\to\infty$ and $k=o(n^{3/4})$ then the expected number of tree components of size $k$
\begin{align}
\gamma_{n,k}(\lambda) & \equiv  \binom{n}{k} k^{k-2} \left( \frac{\lambda}{n} \right)^{k-1}
\left( 1 - \frac{\lambda}{n} \right)^{k(n-k) +\binom{k}{2} - (k-1) }
\label{gencr} \\
& \sim   n \cdot \frac{ k^{-5/2}}{\lambda\sqrt{2\pi}} \exp\left( -\alpha(\lambda)k +
(\lambda-1)\frac{k^2}{2n} - \frac{k^3}{6n^2}\right) 
\label{numcr}
\end{align}
\end{theorem}

\begin{proof}
Cayley showed in 1889 that there are $k^{k-2}$ trees with $k$ labeled vertices. $\binom{n}{k}$ is the number of ways of choosing the $k$ vertices from the $n$ in the graph. For the tree to be there in the ER graph the $k-1$ edges must be present in the graph, there can be no connection between the $k$ vertices and the other $n-k$ vertices, and the other $\binom{k}{2} - (k-1)$ connections between the $k$ vertices must be missing. This proves \eqref{gencr}.

To simplify to get the version in \eqref{numcr} we start by noting that
\begin{align*}
\binom{n}{k} k^{k-2} \cdot  \left( \frac{\lambda}{n} \right)^{k-1}
& \sim n \cdot \frac{(n-1) \ldots (n-k+1)}{n^{k-1}} \cdot \frac{k^{k-2}}{k^k e^{-k} \sqrt{2\pi k}}  \lambda^{k-1}\\
&\sim \sim n \left[ \prod_{j=1}^{k-1} \left( 1 - \frac{j}{n} \right) \right]
\cdot \frac{ k^{-5/2}}{\lambda e^{-k} \sqrt{2\pi}} \lambda^k
\end{align*}
Using the expansion $\log(1-x) = -x - x^2/2 - x^3/3 - \ldots$ we
see that if $k=o(n)$ then
$$
\left( 1 - \frac{\lambda}{n} \right)^{kn - k^2/2} 
\sim \exp(-\lambda k + \lambda k^2/2n )
$$
while if $k=o(n^{3/4})$ we have 
$$ 
\prod_{j=1}^{k-1} \left( 1 - \frac{j}{n} \right) 
 =  \exp\left( - \frac{1}{n}
\sum_{j=1}^{k-1} j - \frac{1}{2n^2} \sum_{j=1}^{k-1} j^2
+ O\left( \frac{k^4}{n^3}\right) \right) \\
 \sim \exp\left( - \frac{k^2}{2n} - \frac{k^3}{6n^2} \right)
$$
Combining the last three calculations gives the desired formula.
\end{proof}

Taking $\lambda=1$, Lemma 2.7.1 implies that the expected number of components of size $k$ in Erd\"os-R\'enyi($n,1/n$) is
$$
\gamma_{n,\lambda}(\lambda_c) \sim \frac{n k^{-5/2}}{\sqrt{2\pi}} e^{-k^3/6n^2}
$$
This says that the largest components are of size $n^{2/3}$.
The critical exponent $\tau$ defined in \eqref{tau}
$$
P(s,t_c) \sim f(0) s^{1-\tau}
$$
has the value $\tau = 5/2$

To get a result  for $\lambda$ close to 1, use the expansion of $\log$
$$
\alpha(\lambda) = \lambda - 1 - \log \lambda  \sim \frac{(\lambda-1)^2}{2}
$$
to get
$$
P(s,t) \approx  \frac{s^{1-5/2}}{\sqrt{2\pi}} \cdot \exp(- s (\lambda-1)^2/2)
$$
Thus \eqref{nearcr} holds with $\tau=5/2$, $\sigma=1/2$ and 
\beq
f(x) = e^{-x/2}
\label{ERsf}
\eeq

\clearp


\begin{thebibliography}{99}


\bibitem{ADSS}
Achlipotas, D., D'Souza, R.M., and Spencer, J. (2009)
Explosive percolation in networks.
{\it Science.} 323, 1453--1455

\bibitem{multcoal}
Aldous, D. (1999)
Brownian excursions, critical random graphs, and multiplicative coalescent.
{\it Ann. Probab.} 25, 812--854

\bibitem{AldousMFC}
Aldous, D. (1999)
Determinisitc and stochastic models for coalescence (aggregation and coagulation):
a review of the mean-file theory for probabilists.
{\it Bernoulli} 5, 3--45

\bibitem{Bastas}
Bastas, N., Giazitzidis, P., Maragakis, M., and Komidis, K. (2014)
Explosive percolation: unusual transition in  a simple model.
{\it Physica A.} 407, 54--65 


\bibitem{BBW}
Bhamidi, S., Budhiraja, A., and Wang, X. (2013)
Aggregation models with limited choice and the multiplicative coalescent.
{\it Prob. Theory Rel. Fields.} 160, 733--796


\bibitem{BohFri}
Bohman, T., and Frieze, A. (2001)
Avoiding a giant component.
{\it Rand. Struct. Alg.} 19, 75--85


\bibitem{BohKra}
Bohman, T., and Kravitz, D. (2006)
Creating a giant component.
{\it Comb. Prob., Comp.} 15, 489--511


\bibitem{daCosta-cont}
da Costa, R.A., Dorogovtsev, S.N., Goltsev, A.V., and Mendes, J.F.F. (2010) 
Explosive percolation transition is actually continuous.
{\it Physical Review Letters.} 105, paper 255701

\bibitem{daCosta-critexp}
da Costa, R.A., Dorogovtsev, S.N., Goltsev, A.V., and Mendes, J.F.F. (2014) 
Critical exponents of the explosive percolation transition.
arXiv:1402.4450


\bibitem{daCosta-scaling}
da Costa, R.A., Dorogovtsev, S.N., Goltsev, A.V., and Mendes, J.F.F. (2014) 
Solution of the explosive percolation quest: Scaling functions and critical exponents.
arXiv:1405.1037


\bibitem{DS-M}
D'Souza, R.M., and Mitzenmacher, M. (2010)
Local cluster aggregation models of explosive percolation.
{\it Physical Review Letters.} 104, paper 195702



\bibitem{DS-N}
D'Souza, R.M., and Nagler, J. (2016)
Anomalous critical and supercritical phenomena in supercritical percolation.
{\it nature Physics.} 11, 531--538

\bibitem{RGD}
Durrett, R. (2007)
{\it Random Graph Dynamics.} Cambridge U. Press

\bibitem{ER60}
Erd\"os, P., and R\'enyi A. (1960)
Publ. Math. Inst. Hungar. Acad. Sci. 5, 17

\bibitem{Grimm}
Grimmett, G. (1999)
{\it Percolation.} Second edition
Springer-Verlag

\bibitem{JanSpe}
Janson, S., and Spencer, J. (2012)
Phase transitions for modified \ER\ processes.
{\it Ark. Mat.} 50, 305--329

\bibitem{KPS}
Kang, M., Perkins, W., and Spencer, J. (2013) 
The Bohman-Frieze process near criticality.
{\it Comb. Prob. Comp.} 43, 221-250

\bibitem{HKSR}
Kesten, H. (1987)
Scaling relations for 2D-percolation
{\it Commun. Math. Phys.} 109, 109-156

\bibitem{Leyvraz}
Leyvraz, F. (2003)
Scaling theory and exactly solved models in the kinets of irreversible aggregation.
{\it Physics Report.} 383, 95--212


\bibitem{RadFor}
Raddicchi, F., and Fortunato, S. (2010)
Explosive percolation: a numerical analysis.
{\it Physical Review E.} 81, paper 036110

\bibitem{RWScience}
Riordan, R., and Warnke, L.  (2011)
Explosive percolation is continuous.
{\it Science.} 333, 322--324 


\bibitem{RWcont}
Riordan, R.,  and Warnke, L. (2012) 
Achlioptas process phase transitions are continuous.
{\it Ann. Appl. Probab.} 22, 1450--1464


\bibitem{RW-ODE}
Riordan, R., and Warnke, L. (2016)
Convergence of Achlioptas processes via differential equations with unique solutions.
{\it Combinatorics, probability, and Computing.} 25, 154--171

\bibitem{RW-bs}
Riordan, R., and Warnke, L. (2017)
The phase transition in bounded-size Achlioptas processes.
arXiv:1704.08714

\bibitem{SabHas}
Sabbir, M.M.H., and Hassan, M.K. (2018)
Product-sum universiality and Rushbrooke inequlaity for explosive percolation.
{\it Physical Review E.} 97, paper 050102(R)

\bibitem{SpeWor}
Spencer, J., and Wormald, N. (2007)
Birth control for giants.
{\it Combinatorica.} 37, 587--638

\bibitem{StauST}
Stauffer, D. (1979)
Scaling theory for percolation clusters.
{\it Physics Reports.} 54, 1--79




\end{thebibliography}
\end{document}